\documentclass[
submission
]{dmtcs-episciences}

\usepackage[utf8]{inputenc}

\usepackage{theorem,enumerate}
\usepackage{amsmath,latexsym,amssymb,amsfonts}
\usepackage{eucal}
\usepackage{color}
\usepackage{comment}

\theorembodyfont{\normalfont\slshape}
\newtheorem{thm}{Theorem}[section]

\newtheorem{prop}[thm]{Proposition}
\newtheorem{lem}[thm]{Lemma}

\newtheorem{claim}{Claim}[section]

\numberwithin{equation}{section}

\def\HH{{ \mathcal{H}}}

\author{Michitaka Furuya\affiliationmark{1}\thanks{This work was supported by JSPS KAKENHI Grant number 26800086 and 18K13449.}}
\title[Formatting an article for DMTCS]{Forbidden subgraphs for constant domination number}
\affiliation{
Kitasato University, Japan}
\keywords{Domination number, Forbidden induced subgraph, Ramsey number}
\received{2018-3-13}
\revised{2018-5-5}
\accepted{2018-5-15}
\begin{document}
\publicationdetails{20}{2018}{1}{19}{4364}
\maketitle
\begin{abstract}
In this paper, we characterize the sets $\mathcal{H}$ of connected graphs such that there exists a constant $c=c(\mathcal{H})$ satisfying $\gamma (G)\leq c$ for every connected $\mathcal{H}$-free graph $G$, where $\gamma (G)$ is the domination number of $G$.
\end{abstract}

\section{Introduction}\label{sec1}

All graphs considered in this paper are finite, simple, and undirected.
Let $G$ be a graph.
Let $V(G)$ and $E(G)$ denote the {\it vertex set} and the {\it edge set} of $G$, respectively.
For a vertex $x\in V(G)$, let $N_{G}(x)$ and $N_{G}[x]$ denote the {\it open neighborhood} and the {\it closed neighborhood}, respectively; thus $N_{G}(x)=\{y\in V(G): xy\in E(G)\}$ and $N_{G}[x]=N_{G}(x)\cup \{x\}$.
For a set $X\subseteq V(G)$, let $N_{G}[X]=\bigcup _{x\in X}N_{G}[x]$.
For a vertex $x\in V(G)$ and a non-negative integer $i$, let $N^{i}_{G}(x)=\{y\in V(G):$ the distance between $x$ and $y$ in $G$ is $i\}$.
Note that $N^{0}_{G}(x)=\{x\}$ and $N^{1}_{G}(x)=N_{G}(x)$.
Let $K_{n}$ and $P_{n}$ denote the {\it complete graph} and the {\it path} of order $n$, respectively.
For terms and symbols not defined in this paper, we refer the reader to \cite{D}.

Let $G$ be a graph.
For two sets $X,Y\subseteq V(G)$, we say that $X$ {\it dominates} $Y$ if $Y\subseteq N_{G}[X]$.
A subset of $V(G)$ which dominates $V(G)$ is called a {\it dominating set} of $G$.
The minimum cardinality of a dominating set of $G$, denoted by $\gamma (G)$, is called the {\it domination number} of $G$.
Since the determining problem of the value $\gamma (G)$ is NP-complete (see \cite{GJ}), many researchers have tried to find good bounds for the domination number (see \cite{HHS1}).
One of the most famous results is due to Ore~\cite{O} who proved that every connected graph $G$ of order at least two satisfies $\gamma (G)\leq |V(G)|/2$.
Here one problem naturally arises:
What additional conditions allow better upper bounds on the domination number?
In this paper, we focus on forbidden induced subgraph conditions.

For a graph $G$ and a set $\HH$ of connected graphs, $G$ is said to be {\it $\HH$-free} if $G$ contains no graph in $\HH$ as an induced subgraph.
In this context, members of $\HH$ are called {\it forbidden subgraphs}.
If $G$ is $\{H\}$-free, then $G$ is simply said to be {\it $H$-free}.
For two sets $\mathcal{H}_{1}$ and $\mathcal{H}_{2}$ of connected graphs, we write $\mathcal{H}_{1}\leq \mathcal{H}_{2}$ if for every $H_{2}\in \mathcal{H}_{2}$, there exists $H_{1}\in \mathcal{H}_{1}$ such that $H_{1}$ is an induced subgraph of $H_{2}$.
The relation ``$\leq $'' between two sets of forbidden subgraphs was introduced in \cite{FKLOPS}.
Note that if $\mathcal{H}_{1}\leq \mathcal{H}_{2}$, then every $\mathcal{H}_{1}$-free graph is also $\mathcal{H}_{2}$-free.

Let $K_{1,3}$ and $K^{*}_{3}$ denote the two unique graphs having degree sequence $(3,1,1,1)$ and $(3,3,3,1,1,1)$, respectively.
Cockayne, Ko and Shepherd~\cite{CKS} (see also Theorem~2.9 in \cite{HHS1}) proved that every connected $\{K_{1,3},K^{*}_{3}\}$-free graph $G$ satisfies $\gamma (G)\leq \lceil |V(G)|/3 \rceil $.
Indeed, Duffus, Gould and Jacobson~\cite{DGJ} proved that every connected $\{K_{1,3},K^{*}_{3}\}$-free graph has a Hamiltonian path.
Since $\gamma (P_{n})=\lceil n/3 \rceil $ for every integer $n$, the above inequality is a consequence of this result.
Furthermore, forbidden induced subgraph conditions for domination-like invariants were widely studied (see, for example, \cite{DRV,DGH,H,JK}).

In this paper, we will characterize the sets $\HH$ of connected graphs satisfying the condition that
\begin{enumerate}
\item[{\bf (A1)}]
there exists a constant $c=c(\HH)$ such that $\gamma (G)\leq c$ for every connected $\HH$-free graph $G$.
\end{enumerate}
Let $n\geq 1$ be an integer.
Let $K^{*}_{n}$ denote the graph with $V(K^{*}_{n})=\{x_{i}:1\leq i\leq n\}\cup \{y_{i}:1\leq i\leq n\}$ and $E(K^{*}_{n})=\{x_{i}x_{j}:1\leq i<j\leq n\}\cup \{x_{i}y_{i}:1\leq i\leq n\}$, and let $S^{*}_{n}$ denote the graph with $V(S^{*}_{n})=\{x\}\cup \{y_{i}:1\leq i\leq n\}\cup \{z_{i}:1\leq i\leq n\}$ and $E(S^{*}_{n})=\{xy_{i}:1\leq i\leq n\}\cup \{y_{i}z_{i}:1\leq i\leq n\}$ (see Figure~\ref{f1}).
Our main result is the following.

\begin{figure}
\begin{center}
{\unitlength 0.1in%
\begin{picture}(45.6000,12.0000)(-1.6000,-16.0000)%
\put(5.2000,-7.0000){\makebox(0,0)[rb]{{\color[named]{Black}{$K^{*}_{n}:$}}}}%
%
{\color[named]{Black}{%
\special{pn 0}%
\special{sh 1.000}%
\special{ia 800 705 50 50 0.0000000 6.2831853}%
}}%
{\color[named]{Black}{%
\special{pn 8}%
\special{ar 800 705 50 50 0.0000000 6.2831853}%
}}%
%
{\color[named]{Black}{%
\special{pn 8}%
\special{ar 1200 1205 600 230 0.0000000 6.2831853}%
}}%
%
{\color[named]{Black}{%
\special{pn 0}%
\special{sh 1.000}%
\special{ia 800 1105 50 50 0.0000000 6.2831853}%
}}%
{\color[named]{Black}{%
\special{pn 8}%
\special{ar 800 1105 50 50 0.0000000 6.2831853}%
}}%
%
{\color[named]{Black}{%
\special{pn 0}%
\special{sh 1.000}%
\special{ia 1000 1105 50 50 0.0000000 6.2831853}%
}}%
{\color[named]{Black}{%
\special{pn 8}%
\special{ar 1000 1105 50 50 0.0000000 6.2831853}%
}}%
%
{\color[named]{Black}{%
\special{pn 0}%
\special{sh 1.000}%
\special{ia 1600 1105 50 50 0.0000000 6.2831853}%
}}%
{\color[named]{Black}{%
\special{pn 8}%
\special{ar 1600 1105 50 50 0.0000000 6.2831853}%
}}%
%
{\color[named]{Black}{%
\special{pn 4}%
\special{sh 1}%
\special{ar 1200 1105 16 16 0 6.2831853}%
\special{sh 1}%
\special{ar 1400 1105 16 16 0 6.2831853}%
\special{sh 1}%
\special{ar 1300 1105 16 16 0 6.2831853}%
\special{sh 1}%
\special{ar 1300 1105 16 16 0 6.2831853}%
}}%
%
{\color[named]{Black}{%
\special{pn 0}%
\special{sh 1.000}%
\special{ia 1000 705 50 50 0.0000000 6.2831853}%
}}%
{\color[named]{Black}{%
\special{pn 8}%
\special{ar 1000 705 50 50 0.0000000 6.2831853}%
}}%
%
{\color[named]{Black}{%
\special{pn 0}%
\special{sh 1.000}%
\special{ia 1600 705 50 50 0.0000000 6.2831853}%
}}%
{\color[named]{Black}{%
\special{pn 8}%
\special{ar 1600 705 50 50 0.0000000 6.2831853}%
}}%
%
{\color[named]{Black}{%
\special{pn 8}%
\special{pa 1600 705}%
\special{pa 1600 1105}%
\special{fp}%
\special{pa 1000 1105}%
\special{pa 1000 705}%
\special{fp}%
\special{pa 800 705}%
\special{pa 800 1105}%
\special{fp}%
}}%
\put(8.0000,-12.4500){\makebox(0,0){{\color[named]{Black}{$x_{1}$}}}}%
\put(10.0000,-12.4500){\makebox(0,0){{\color[named]{Black}{$x_{2}$}}}}%
\put(16.0000,-12.4500){\makebox(0,0){{\color[named]{Black}{$x_{n}$}}}}%
\put(16.0000,-5.5500){\makebox(0,0){{\color[named]{Black}{$y_{n}$}}}}%
\put(10.0000,-5.5500){\makebox(0,0){{\color[named]{Black}{$y_{2}$}}}}%
\put(8.0000,-5.5500){\makebox(0,0){{\color[named]{Black}{$y_{1}$}}}}%
%
{\color[named]{Black}{%
\special{pn 0}%
\special{sh 1.000}%
\special{ia 2800 705 50 50 0.0000000 6.2831853}%
}}%
{\color[named]{Black}{%
\special{pn 8}%
\special{ar 2800 705 50 50 0.0000000 6.2831853}%
}}%
%
{\color[named]{Black}{%
\special{pn 0}%
\special{sh 1.000}%
\special{ia 2800 1105 50 50 0.0000000 6.2831853}%
}}%
{\color[named]{Black}{%
\special{pn 8}%
\special{ar 2800 1105 50 50 0.0000000 6.2831853}%
}}%
%
{\color[named]{Black}{%
\special{pn 0}%
\special{sh 1.000}%
\special{ia 3000 1105 50 50 0.0000000 6.2831853}%
}}%
{\color[named]{Black}{%
\special{pn 8}%
\special{ar 3000 1105 50 50 0.0000000 6.2831853}%
}}%
%
{\color[named]{Black}{%
\special{pn 0}%
\special{sh 1.000}%
\special{ia 3000 705 50 50 0.0000000 6.2831853}%
}}%
{\color[named]{Black}{%
\special{pn 8}%
\special{ar 3000 705 50 50 0.0000000 6.2831853}%
}}%
%
{\color[named]{Black}{%
\special{pn 0}%
\special{sh 1.000}%
\special{ia 3600 705 50 50 0.0000000 6.2831853}%
}}%
{\color[named]{Black}{%
\special{pn 8}%
\special{ar 3600 705 50 50 0.0000000 6.2831853}%
}}%
%
{\color[named]{Black}{%
\special{pn 0}%
\special{sh 1.000}%
\special{ia 3600 1105 50 50 0.0000000 6.2831853}%
}}%
{\color[named]{Black}{%
\special{pn 8}%
\special{ar 3600 1105 50 50 0.0000000 6.2831853}%
}}%
%
{\color[named]{Black}{%
\special{pn 0}%
\special{sh 1.000}%
\special{ia 3200 1505 50 50 0.0000000 6.2831853}%
}}%
{\color[named]{Black}{%
\special{pn 8}%
\special{ar 3200 1505 50 50 0.0000000 6.2831853}%
}}%
%
{\color[named]{Black}{%
\special{pn 8}%
\special{pa 3200 1505}%
\special{pa 2800 1105}%
\special{fp}%
\special{pa 2800 1105}%
\special{pa 2800 705}%
\special{fp}%
\special{pa 3000 705}%
\special{pa 3000 705}%
\special{fp}%
\special{pa 3000 705}%
\special{pa 3000 1105}%
\special{fp}%
\special{pa 3000 1105}%
\special{pa 3200 1505}%
\special{fp}%
\special{pa 3200 1505}%
\special{pa 3600 1105}%
\special{fp}%
\special{pa 3600 1105}%
\special{pa 3600 705}%
\special{fp}%
}}%
%
{\color[named]{Black}{%
\special{pn 4}%
\special{sh 1}%
\special{ar 3200 910 16 16 0 6.2831853}%
\special{sh 1}%
\special{ar 3400 910 16 16 0 6.2831853}%
\special{sh 1}%
\special{ar 3300 910 16 16 0 6.2831853}%
\special{sh 1}%
\special{ar 3300 910 16 16 0 6.2831853}%
}}%
\put(33.3000,-15.0000){\makebox(0,0){{\color[named]{Black}{$x$}}}}%
\put(26.7000,-11.1000){\makebox(0,0){{\color[named]{Black}{$y_{1}$}}}}%
\put(31.5000,-11.1000){\makebox(0,0){{\color[named]{Black}{$y_{2}$}}}}%
\put(37.5000,-11.1000){\makebox(0,0){{\color[named]{Black}{$y_{n}$}}}}%
\put(36.0000,-5.5500){\makebox(0,0){{\color[named]{Black}{$z_{n}$}}}}%
\put(30.0000,-5.5500){\makebox(0,0){{\color[named]{Black}{$z_{2}$}}}}%
\put(28.0000,-5.5500){\makebox(0,0){{\color[named]{Black}{$z_{1}$}}}}%
\put(17.6000,-14.3000){\makebox(0,0){{\color[named]{Black}{$K_{n}$}}}}%
\put(25.2000,-7.0000){\makebox(0,0)[rb]{{\color[named]{Black}{$S^{*}_{n}:$}}}}%
%
{\color[named]{Black}{%
\special{pn 4}%
\special{pa 4000 400}%
\special{pa 4400 400}%
\special{pa 4400 1600}%
\special{pa 4000 1600}%
\special{pa 4000 400}%
\special{ip}%
}}%
\end{picture}}%

\caption{Graphs $K^{*}_{n}$ and $S^{*}_{n}$}
\label{f1}
\end{center}
\end{figure}
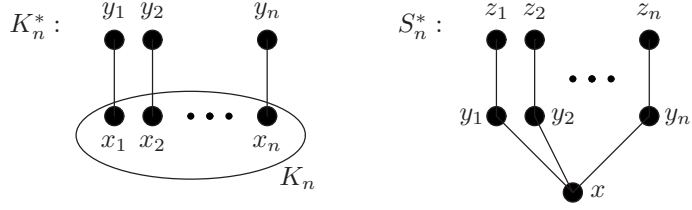

\begin{thm}
\label{thm1}
Let $\HH$ be a set of connected graphs.
Then $\HH$ satisfies (A1) if and only if $\HH\leq \{K^{*}_{k},S^{*}_{\ell},P_{m}\}$ for some positive integers $k$, $\ell$ and $m$.
\end{thm}

We conclude this section by considering the case where a set $\HH$ can contain disconnected graphs.
Then the following proposition holds.

\begin{prop}
Let $\HH$ be a set of graphs.
Then $\HH$ satisfies (A1) if and only if $\HH\leq \{\overline{K_{k}}\}$ for some positive integer $k$.
\end{prop}
\begin{proof}
Suppose that $\HH$ satisfies (A1).
Then there exists a constant $c=c(\HH)$ such that $\gamma (G)\leq c$ for every connected $\HH$-free graph $G$.
Since $\gamma (\overline{K_{c+1}})=c+1$, $\overline{K_{c+1}}$ is not $\HH$-free, and so $\HH\leq \{\overline{K_{c+1}}\}$.

On the other hand, if $\HH\leq \{\overline{K_{k}}\}$, then every $\HH$-free graph $G$ satisfies $\gamma (G)\leq k-1$ because every maximal independent set of $G$ is a dominating set.
\end{proof}

\section{Proof of Theorem~\ref{thm1}}\label{sec2}

For positive integers $s$ and $t$, let $R(s,t)$ denote the {\it Ramsey number} with respect to $s$ and $t$.
For positive integers $k$, $\ell$ and $i$, we recursively define $g_{k,\ell}(i)$ as follows:
$$
\begin{cases}
g_{k,\ell}(1) = 1\\
g_{k,\ell}(i) = R(k,(\ell-1)g_{k,\ell}(i-1)+1)-1~~~~(i\geq 2).
\end{cases}
$$

\begin{lem}
\label{lem1.1}
Let $k$, $\ell$ and $i$ be positive integers.
Let $G$ be a $\{K^{*}_{k},S^{*}_{\ell}\}$-free graph, and let $a$ be a vertex of $G$.
Then for an independent set $X\subseteq N^{i}_{G}(a)$, there exists $U\subseteq N^{i-1}_{G}(a)$ with $|U|\leq g_{k,\ell}(i)$ that dominates $X$.
\end{lem}
\begin{proof}
We proceed by induction on $i$.
If $i=1$, then $U=\{a\}$ is a desired subset of $N^{i-1}_{G}(a)=\{a\}$.
Thus we may assume that $i\geq 2$.
Note that $N^{i-1}_{G}(a)$ dominates $X$.
Let $U$ be a minimal subset of $N^{i-1}_{G}(a)$ that dominates $X$.
It suffices to show that $|U|\leq R(k,(\ell-1)g_{k,\ell}(i-1)+1)-1=g_{k,\ell}(i)$.

By way of contradiction, suppose that $|U|\geq R(k,(\ell-1)g_{k,\ell}(i-1)+1)$.
For each $u\in U$, since $U-\{u\}$ does not dominate $X$ by the minimality of $U$, there exists a vertex $x_{u}\in X$ such that $N_{G}(x_{u})\cap U=\{u\}$.
Recall that $X$ is an independent set.
If there exists a clique $U_{1}\subseteq U$ with $|U_{1}|=k$, then the subgraph of $G$ induced by $U_{1}\cup \{x_{u}:u\in U_{1}\}$ is isomorphic to $K^{*}_{k}$, which contradicts the $K^{*}_{k}$-freeness of $G$.
Since $|U|\geq R(k,(\ell-1)g_{k,\ell}(i-1)+1)$, this implies that there exists an independent set $U_{2}\subseteq U$ with $|U_{2}|=(\ell-1)g_{k,\ell}(i-1)+1$.
By the induction hypothesis, there exists $U'\subseteq N^{i-2}_{G}(a)$ with $|U'|=g_{k,\ell}(i-1)$ that dominates $U_{2}$.
By the pigeon-hole principle, there exists a vertex $u'\in U'$ such that $|N_{G}(u')\cap U_{2}|\geq \ell$.
Let $\tilde{U}_{2}\subseteq N_{G}(u')\cap U_{2}$ be a set with $|\tilde{U}_{2}|=\ell$.
Then the subgraph of $G$ induced by $\{u'\}\cup \tilde{U}_{2}\cup \{x_{u}:u\in \tilde{U}_{2}\}$ is isomorphic to $S^{*}_{\ell}$, which is a contradiction.
\end{proof}

For positive integers $k$, $\ell$ and $i$ with $i\geq 2$, let $f_{k,\ell}(i) = R(k,\ell)g_{k,\ell}(i)$.

\begin{lem}
\label{lem1.2}
Let $k$, $\ell$ and $i$ be positive integers with $i\geq 2$.
Let $G$ be a $\{K^{*}_{k},S^{*}_{\ell}\}$-free graph, and let $a$ be a vertex of $G$.
Then there exists $\hat{U}\subseteq V(G)$ with $|\hat{U}|\leq f_{k,\ell}(i)$ that dominates $N^{i}_{G}(a)$.
\end{lem}
\begin{proof}
Let $X$ be a maximal independent subset of $N^{i}_{G}(a)$.
By Lemma~\ref{lem1.1}, there exists $U\subseteq N^{i-1}_{G}(a)$ with $|U|\leq g_{k,\ell}(i)$ that dominates $X$.
By the maximality of $X$, $X$ dominates $N^{i}_{G}(a)$, and so $X$ dominates $N^{i}_{G}(a)-N_{G}[U]$.
Let $X_{0}$ be a minimal subset of $X$ that dominates $N^{i}_{G}(a)-N_{G}[U]$.

\begin{claim}
\label{cl-lem1.2-1}
We have $|X_{0}|\leq (R(k,\ell)-1)g_{k,\ell}(i)$.
\end{claim}
\begin{proof}
Suppose that $|X_{0}|\geq (R(k,\ell)-1)g_{k,\ell}(i)+1$.
Since $U$ dominates $X_{0}$ and $|U|\leq g_{k,\ell}(i)$, there exists a vertex $u'\in U$ such that $|N_{G}(u')\cap X_{0}|\geq R(k,\ell)$.
For each $x\in X_{0}$, since $X_{0}-\{x\}$ does not dominate $N^{i}_{G}(a)-N_{G}[U]$ by the minimality of $X_{0}$, there exists a vertex $y_{x}\in N^{i}_{G}(a)-N_{G}[U]$ such that $N_{G}(y_{x})\cap X_{0}=\{x\}$.
Set $Y=\{y_{x}:x\in N_{G}(u')\cap X_{0}\}$, and for each $y\in Y$, write $N_{G}(y)\cap X_{0}=\{x_{y}\}$.
Note that $\{x_{y}:y\in Y\}\subseteq N_{G}(u')\cap X_{0}$ and $y_{x_{y}}=y$ for each $y\in Y$.
Since $|Y|=|N_{G}(u')\cap X_{0}|\geq R(k,\ell)$, there exists a clique $Y_{1}\subseteq Y$ with $|Y_{1}|=k$ or an independent set $Y_{2}\subseteq Y$ with $|Y_{2}|=\ell$.
Recall that $Y\subseteq N^{i}_{G}(a)-N_{G}[U]$, and so $N_{G}(u')\cap Y=\emptyset $.
If there exists a clique $Y_{1}\subseteq Y$ with $|Y_{1}|=k$, then the subgraph of $G$ induced by $Y_{1}\cup \{x_{y}:y\in Y_{1}\}$ is isomorphic to $K^{*}_{k}$; if there exists an independent set $Y_{2}\subseteq Y$ with $|Y_{2}|=\ell$, then the subgraph of $G$ induced by $\{u'\}\cup \{x_{y}:y\in Y_{2}\} \cup Y_{2}$ is isomorphic to $S^{*}_{\ell}$.
In either case, we obtain a contradiction.
\end{proof}

Recall that $X_{0}$ dominates $N^{i}_{G}(a)-N_{G}[U]$.
Hence $U\cup X_{0}$ dominates $N^{i}_{G}(a)$.
Furthermore, by the definition of $U$ and Claim~\ref{cl-lem1.2-1},
$$
|U\cup X_{0}|=|U|+|X_{0}|\leq g_{k,\ell}(i)+(R(k,\ell)-1)g_{k,\ell}(i)=f_{k,\ell}(i).
$$
Thus $\hat{U}=U\cup X_{0}$ is a desired set.
\end{proof}

\begin{proof}[of Theorem~\ref{thm1}]
We first prove the ``only if''part.
Let $\HH$ be a set of connected graphs satisfying (A1).
Then there exists a constant $c=c(\HH)$ such that $\gamma (G)\leq c$ for every connected $\HH$-free graph $G$.
Since we can easily verify that $\gamma (K^{*}_{c+1})=\gamma (S^{*}_{c+1})=\gamma (P_{3c+1})=c+1$, none of $K^{*}_{c+1}$, $S^{*}_{c+1}$ and $P_{3c+1}$ is $\HH$-free.
This implies that $\HH\leq \{K^{*}_{c+1},S^{*}_{c+1},P_{3c+1}\}$, as desired.

Next we prove the ``if'' part.
Let $\HH$ be a set of connected graphs such that $\HH\leq \{K^{*}_{k},S^{*}_{\ell},P_{m}\}$ for some positive integers $k$, $\ell$ and $m$.
Choose $k$, $\ell$ and $m$ so that $k+\ell+m$ is as small as possible.
Then $k$, $\ell$ and $m$ are uniquely determined.
In particular, the value $1+\sum _{2\leq i\leq m-2}f_{k,\ell}(i)$ only depends on $\HH$.
Furthermore, every $\HH$-free graph is also $\{K^{*}_{k},S^{*}_{\ell},P_{m}\}$-free.
Thus it suffices to show that every connected $\{K^{*}_{k},S^{*}_{\ell},P_{m}\}$-free graph $G$ satisfies $\gamma (G)\leq 1+\sum _{2\leq i\leq m-2}f_{k,\ell}(i)$.
Let $a\in V(G)$.
Since $G$ is $P_{m}$-free, $N^{i}_{G}(a)=\emptyset $ for all $i\geq m-1$.
Since $G$ is connected, this implies that $V(G)=\bigcup _{0\leq i\leq m-2}N^{i}_{G}(a)$.
Since $G$ is $\{K^{*}_{k},S^{*}_{\ell}\}$-free, it follows from Lemma~\ref{lem1.2} that for each $i$ with $2\leq i\leq m-2$, there exists a set $\hat{U}_{i}\subseteq V(G)$ with $|\hat{U}_{i}|\leq f_{k,\ell}(i)$ that dominates $N^{i}_{G}(a)$.
Since $\{a\}$ dominates $N^{0}_{G}(a)\cup N^{1}_{G}(a)$, $\{a\}\cup (\bigcup _{2\leq i\leq m-2}\hat{U}_{i})$ is a dominating set of $G$, and so
$$
\gamma (G)\leq |\{a\}|+\sum _{2\leq i\leq m-2}|\hat{U}_{i}|\leq 1+\sum _{2\leq i\leq m-2}f_{k,\ell}(i),
$$
as desired.

This completes the proof of Theorem~\ref{thm1}.
\end{proof}

\section{Concluding remark}\label{sec3}

In this paper, we characterized the sets $\HH$ of connected graphs satisfying (A1).
For similar problems concerning many domination-like invariants, we can use the sets appearing in Theorem~\ref{thm1}.

Let $\mu $ be an invariant of graphs, and assume that
\begin{enumerate}
\item[{\bf (D1)}]
there exist two constants $c_{1},c_{2}\in \mathbb{R}^{+}$ such that $c_{1}\gamma (G)\leq \mu (G)\leq c_{2}\gamma (G)$ for all connected graphs $G$.
\end{enumerate}
Note that many important domination-like invariants (for example, total domination number $\gamma _{t}$, paired domination number $\gamma _{\rm pr}$, Roman domination number $\gamma _{R}$, rainbow domination number $\gamma _{{\rm r}k}$, etc.) satisfy (D1).
Furthermore, we focus on the condition that
\begin{enumerate}
\item[{\bf (A'1)}]
there exists a constant $c'=c'(\mu,\HH)$ such that $\mu (G)\leq c$ for every connected $\HH$-free graph $G$.
\end{enumerate}

We first suppose that a set $\HH$ of connected graphs satisfies (A'1).
Note that
\begin{enumerate}[$\bullet $]
\item
$\mu (K^{*}_{\lceil (c'+1)/c_{1} \rceil })\geq c_{1}\gamma (K^{*}_{\lceil (c'+1)/c_{1} \rceil })=c_{1}\cdot \lceil (c'+1)/c_{1} \rceil \geq c'+1$,
\item
$\mu (S^{*}_{\lceil (c'+1)/c_{1} \rceil })\geq c_{1}\gamma (S^{*}_{\lceil (c'+1)/c_{1} \rceil })=c_{1}\cdot \lceil (c'+1)/c_{1} \rceil \geq c'+1$, and
\item
$\mu (P_{3\lceil (c'+1)/c_{1} \rceil +1})\geq c_{1}\gamma (P_{3\lceil (c'+1)/c_{1} \rceil +1})=c_{1}\cdot \lceil (c'+1)/c_{1} \rceil \geq c'+1$.
\end{enumerate}
Thus, by similar argument to the proof of ``only if'' part of Theorem~\ref{thm1}, we have $\HH\leq \{K^{*}_{k},S^{*}_{\ell},P_{m}\}$ for some positive integers $k$, $\ell$ and $m$.

On the contrary, suppose that a set $\HH$ of connected graphs satisfies $\HH\leq \{K^{*}_{k},S^{*}_{\ell},P_{m}\}$ for some positive integers $k$, $\ell$ and $m$.
Then by Theorem~\ref{thm1}, (A1) holds, and hence for a connected $\HH$-free graph $G$, we have
$$
\mu (G)\leq c_{2}\gamma (G)\leq c_{2}\cdot c(\HH).
$$
Consequently (A'1) holds (for $c'=c_{2}\cdot c(\HH)$).
Therefore, we obtain the following theorem.

\begin{thm}
\label{thm2}
Let $\mu $ be an invariant for graphs satisfying (D1), and let $\HH$ be a set of connected graphs.
Then $\HH$ satisfies (A'1) if and only if $\HH\leq \{K^{*}_{k},S^{*}_{\ell},P_{m}\}$ for positive integers $k$, $\ell$ and $m$.
\end{thm}

\acknowledgements
I would like to thank anonymous referees for careful reading and helpful comments.


\begin{thebibliography}{99}
\bibitem{CKS}
E.J.~Cockayne, C.W.~Ko and F.B.~Shepherd,
Inequalities concerning dominating sets in graphs,
Technical Report DM-370-IR, Dept. Math., Univ. Victoria (1985).

\bibitem{DRV}
P.~Dankelmann, D.~Rautenbach and L.~Volkmann,
Weighted domination in triangle-free graphs,
Discrete Math. {\bf 250} (2002) 233--239.

\bibitem{D}
R.~Diestel,
{\it Graph Theory} (5th edition), Graduate Texts in Mathematics {\bf 173},
Springer (2017).

\bibitem{DGH}
P.~Dorbec, S.~Gravier and M.A.~Henning,
Paired-domination in generalized claw-free graphs,
J. Comb. Optim. {\bf 14} (2007) 1--7.

\bibitem{DGJ}
D.~Duffus, R.J.~Gould and M.S.~Jacobson,
Forbidden subgraphs and the Hamiltonian theme,
{\it The theory and applications of graphs}, pp.297--316, Wiley (1981).

\bibitem{FKLOPS}
S.~Fujita, K.~Kawarabayashi, C.~L.~Lucchesi, K.~Ota, M.~Plummer and A.~Saito,
A pair of forbidden subgraphs and perfect matchings,
J. Combin. Theory Ser. B {\bf 96} (2006) 315--324.

\bibitem{GJ}
M.R.~Garey and D.S.~Johnson,
``Computers and Intractability: A Guide to the Theory of NP-Completeness'',
Freeman, New York (1979).

\bibitem{H}
J.~Haviland,
Independent domination in triangle-free graphs, 
Discrete Math. {\bf 308} (2008) 3545--3550.

\bibitem{HHS1}
T.W.~Haynes, S.T.~Hedetniemi and P.J.~Slater,
Fundamentals of Domination in Graphs,
Marcel Dekker, Inc. New York (1998).


\bibitem{JK}
H.~Jiang and L.~Kang,
Total restrained domination in claw-free graphs,
J. Comb. Optim. {\bf 19} (2010) 60--68.

\bibitem{O}
O.~Ore,
Theory of graphs,
American Mathematical Society Colloquium Publications Vol.38 American Mathematics Society, Providence, RI (1962).


\end{thebibliography}
\end{document}